\documentclass[12pt,draftcls,onecolumn]{IEEEtran}
\usepackage{fullpage}
\usepackage{latexsym,url}
\usepackage{amsmath,amssymb,array,graphicx,verbatim}
\usepackage{setspace}
\usepackage
  [breaklinks,bookmarks,bookmarksnumbered,bookmarksopen,bookmarksopenlevel=2]
  {hyperref}

\newcommand{\eg}{{\it e.g.}}

\newtheorem{theorem}{Theorem}
\newtheorem{lemma}[theorem]{Lemma}

\begin{document}

\title{\LARGE{Power Allocation and Spectrum Sharing in Multi-User, Multi-Channel
Systems with Strategic Users}}
\author{\Large{Ali~Kakhbod and Demosthenis~Teneketzis} \\
\large{Department of Electrical Engineering and Computer Science,} \\
\large{University of Michigan, Ann Arbor, MI, USA.} \\
\normalsize{Email: {\tt{\{akakhbod,teneket\}@umich.edu}}}} 

%\author{Joseph C. Koo}
%\author{Demosthenis Teneketzis}
%\runauthor{A. Kakhbod, J. C. Koo and D. Teneketzis}

%\address{}
%\address{Department of Electrical Engineering, Stanford University,
%Stanford, CA  94305, U.S.A.; E-mail: jckoo@stanford.edu}
%\address{Department of Electrical Engineering and Computer Science,
%University of Michigan, Ann Arbor, MI  48109, U.S.A.; E-mail:
%teneket@umich.edu}

\maketitle
\begin{abstract}
We consider the decentralized power allocation and spectrum sharing problem in multi-user, multi-channel systems
with strategic users. We present a mechanism/game form that has the following desirable
features. (1) It is individually rational. (2) It is  budget balanced at every Nash equilibrium of the game induced by the game form as well as off equilibrium.
(3) The allocation corresponding to  every Nash equilibrium (NE) of the game induced by the mechanism is a Lindahl allocation, that is, a weakly Pareto optimal allocation.  Our proposed game form/mechanism achieves all the above desirable properties without any
assumption about, concavity, differentiability, monotonicity, or quasi-linearity of the users' utility functions. 
\end{abstract}
\begin{section}{Introduction}
\begin{subsection}{Motivation and Challenges}
As wireless communication devices become more pervasive,
the demand for the frequency spectrum that serves as
the underlying medium grows.
%Traditionally, the problem of
%allocating the resource of the frequency spectrum has been
%handled by granting organizations and companies licenses to
%broadcast at certain frequencies. This rigid approach leads
%to significant under-utilization of this scarce resource. Moreover, frequency utilization varies significantly
%with time and location \cite{akildyz}. A cognitive radio is a wireless communication device that
%is aware of its capabilities, environment, and intended use,
%and can also learn new waveforms, models, or operational
%scenarios \cite{neel}.
Recently, the Federal Communications Commission (FCC) has established rules (see \cite{fcc}) that describes how cognitive radios can lead to more efficient use of the frequency  spectrum. These rules along with the cognitive radio's features and the fact that information in the wireless network is decentralized  and users may be strategic give rise to a wealth of important and challenging  research issues  associated with power allocation and spectrum sharing. These issues  have recently attracted a lot of interest (\eg \; \cite{tse}, \cite{berry}, \cite{liu} and \cite{cioffi}; a more detailed discussion of the references and their comparison with the results of this paper will be presented in section \ref{dis}). 

In this paper we investigate a power allocation and spectrum sharing problem arising in multi-user, multi-channel systems with decentralized information and strategic/selfish users. We formulate the problem as a public good  (\cite{tian} Ch. 12) allocation with strategic users. We propose an approach based on the philosophy of mechanism design, in particular, implementation theory (\cite{tian} Ch. 15). We present a game form /mechanism (\cite{tian} Ch. 15.2.3) and analyze its properties.  We compare our results with those already available in the literature.  
%\begin{comment}
\end{subsection}
\begin{comment}
\begin{subsection}{Contribution of the paper}
\label{1212}
\textit{The main contribution of this paper is the discovery of a decentralized  mechanism/game form for power allocation and spectrum sharing in multi-user, multi-channel systems with strategic users, which possesses the following properties.} 
\begin{itemize}
	\item \textit{(P1) The allocation corresponding to  every Nash equilibrium (NE) of
the game induced by the game form/mechanism results in a Lindahl equilibrium, that is, it is weakly Pareto optimal. Conversely, every Lindahl equilibrium results in a NE of the game induced by the proposed game form/mechanism.}
 \item \textit{(P2) The mechanism is individually rational, i.e., every user participates voluntarily in the game induced by the
game form.}
 \item \textit{(P3) The mechanism results in a balanced budget at every NE of the game induced by the game form  as well as
off equilibrium.}
\end{itemize}
\textit{All the above desirable properties are achieved without any 
assumption about, concavity, differentiability, monotonicity or quasi-linearity of the users' utility functions.  
}\\
We compare our contributions with the existing literature in section \ref{dis} of this paper, after
we present and prove our results.
\end{subsection}
\end{comment}

\begin{subsection}{Organization of the paper}
The rest of the paper is organized as follows.  In section \ref{00} we present our model, describe the assumptions  on the model's information structure  and state our objective. In section \ref{11} we describe the allocation game form/mechanism we propose for the solution of our problem. In section \ref{2222} we interpret the components of the proposed game form/mechanism.  In section \ref{33} we investigate the properties of the proposed game form. In section \ref{dis} we  compare the results of this paper with those of the existing literature. We conclude in section \ref{con}. 

\end{subsection}
\end{section}
\begin{section}{The Model and Objective}
\label{00}
\begin{subsection}{The Model}
We consider $N$ users/agents communicating over $f$ frequency bands. Let $\textbf{N}:=\{1,\cdots,N\}$ be the set of users, and $\textbf{F}:=\{1,2,\cdots,f\}$  the set of frequency bands. Each user $i, i\in \textbf{N},$ is a communicating pair consisting of one transmitter and one receiver. There is one additional agent, the  $(N+1)^{th}$ agent, who is different from all the other $N$ agents/users and whose role will be described below. Each user has a fixed total power $\bar{W}$ which he can allocate over the set $\textbf{F}$ of frequency bands. Let $p_i^j$, $i \in \textbf{N}, j \in \textbf{F}$ denote  the power user $i$  allocates to frequency band $j$. The power $p_i^j, i \in \textbf{N}, j \in \textbf{F}$ must be chosen from the set $\textbf{Q}:=\{0,Q_1,Q_2,\cdots,Q_l\}$ where $Q_k>0, 1 \leq k \leq l$ and $0$ means that user $i$ does not use frequency band $j \in \textbf{F}$ to communicate information. In other words,  \textbf{Q} is a set of quantization levels that a user can use when he allocates power in a certain frequency band. Let $\bar{p}_i=(p_i^1,p_i^2,\cdots,p_i^f), i\in \textbf{N}$, denote a feasible bundle of power user $i$ allocates over the frequency bands in $\textbf{F}$. That is, $p_i^j\in \textbf{Q}, \ \forall j\in \textbf{F},$ and $\sum_{j\in\textbf{F}}p_i^j\leq \bar{W}$. Let $\texttt{P}=(\bar{p}_1,\bar{p}_2,\cdots,\bar{p}_N)$ be a profile of feasible bundles of powers allocated by the $N$ users over the frequency bands in $\textbf{F}$; let  $\Pi$  denote the set of all feasible profiles $\texttt{P}$. Since the sets, $\textbf{N},\textbf{F}$ and $\textbf{Q}$ are finite, $\Pi$ is finite. Let $|\Pi|=G_N$; we represent  every feasible power profile by a number between $1$ and $G_N$.  Thus, $\Pi=\{1,2,\cdots,{G_N}\}$. If user $i$ allocates positive power in frequency band $j$, it may experience interference from those users who also allocate positive power in that frequency band. The intensity of the interference experienced by user $i, i\in \textbf{N},$ depends on the power profiles used by the other users and the `channel gains' $h_{ji}$ between the other users $j, j\neq i,$ and $i$. The satisfaction that user $i, i\in \textbf{N},$ obtains during the communication process depends on his transmission power and the intensity of the interference he experiences. Consequently, user $i$'s, $i \in \textbf{N}$, satisfaction depends on the whole feasible bundle $k, k\in \Pi,$ of power and is described by his utility function $V_i(k,t_i)$, $i \in \textbf{N}$, where $t_i \in \mathbb{R}$ represents the tax (subsidy) user $i$ pays (receives) for communicating. One  example of such a utility function is presented  in the discussion following the assumptions. All taxes are paid to the $(N+1)^{th}$ agent who is not a profit maker; this agent acts like an accountant, collects the money from all users who pay taxes and redistributes it to all users who receive subsidies. \\
  We now state our assumptions about the model, the users' utility functions, and the nature of the problem we investigate. Some of these assumptions are restrictions we impose, some others are a consequence of the nature of the problem we investigate. We comment on each of the assumptions we make after we state them.\\
%\begin{itemize}
\textbf{(A1)}: We consider a static power allocation and spectrum sharing problem.\\
\textbf{(A2)}: Each agent/user is aware of all the other users present in the system. Users talk to each other and exchange messages in a broadcast setting. That is, each user hears every other user's message; the $(N+1)^{th}$ agent hears all the other users' messages. After the message exchange process ends/converges, decision about power allocations at various frequency bands are made.\\ 
\textbf{(A3)}: Each user's transmission at a particular frequency band creates interference to every user transmitting in the same frequency band. \\
\textbf{(A4)}: The channel gains $h_{ji}(\hat{f}), j,i \in \textbf{N},  \hat{f}\in\textbf{F}$ are known to user $i, i\in \textbf{N}$. The gains $h_{ji}(\hat{f})$, $j, i \in \textbf{N}, \ \hat{f}\in\textbf{F}$, do not change during the communication process.\\     
\textbf{(A5)}: Each user's utility $V_i(x,t_i), x \in \Pi \cup \{0\}$\footnote{ The number zero denotes every non-feasible allocation.}, is decreasing in $t_i, t_i \in \mathbb{R}$, $i \in \textbf{N}$. Furthermore,  $V_i(x,t_i)\geq V_i(0,t_i)$ for any $t_i \in \mathbb{R}$ and $x \in \Pi$. \\
\textbf{(A6)}: The utility function $V_i, i\in \textbf{N},$ is user $i$'s private information.\\
\textbf{(A7)}: The quantization set  $\textbf{Q}$ is selected from $\mathcal{Q}$; the parameter $\bar{W}$ is selected from $\mathcal{W}$, and $V_i$ is selected from $\mathcal{V}$ for all $i, i \in \textbf{N}$.  $\textbf{Q}, \mathcal{Q}, \bar{W}, \mathcal{W}$ and $\mathcal{V}$ are common knowledge among all users.\\
\textbf{(A8)}: Each user behaves strategically, that is, each user is selfish and attempts to maximize his own utility function under the constraints on the total power available to him, and the set $\textbf{Q}$ of quantization levels.\\
%(i.e., each user is selfish, his objective is to maximize his own utility function $V_i(\texttt{P},t_i)$).
%	\item (A3): Every user knows $\textbf{Q}$ and  $\bar{W}$ and it is common knowledge \cite{fudenberg} among the users. 
%\item (A9): Users talk to each other and exchange messages in a broadcast setting, i.e., each user hears every other user's message.
%\item (A5): The coordinator/regulator is \textit{not} a profit maker.
%	There is a "\textit{planner}/\textsl{owner}" who is not profit maker. It is assumed that the planner requests  certain amount of money denoted by $\delta, \delta \geq 0$, from the network (all the users) for his services, i.e., $\sum_{i=1}^N t_i=\delta$ 
\textbf{(A9)}: The representation/association of every feasible power profile by a number in the set $\Pi$ is common knowledge among all users. \\ 
%\end{itemize}
We now briefly discuss each of the above assumptions. 
%We would like to solve the dynamic power allocation and spectrum sharing problem with strategic users. That is, we would like to solve the problem where strategic %cognitive radio users accumulate (over time) knowledge about the network and use this knowledge to dynamically adjust the power allocated to various frequency bands so as %to efficiently utilize the available spectrum.   Currently, we are unable to tackle this problem, thus, we restrict attention to the static one (\textbf{(A1)}).
We restrict attention to the static power allocation and spectrum sharing problem (\textbf{(A1)}). The dynamic problem is a major open problem that we intend to address in the future.  We assume that all users are in a relatively small area,  so they can hear each other, are aware of the presence of one another, interfere with one another and exchange messages in a broadcast setting (\textbf{(A2)},\textbf{(A3)}). 
%The situation where each user affects a subset of the set $\textbf{N}$ of users present in the system, and is affected by another  subset of $\textbf{N}$ will be addressed in a future publication. 
Since each user's satisfaction depends on his transmission power and the interference he experiences, his utility will depend on the whole power profile $x\in \Pi$; furthermore, the higher the tax a user pays, the lower is his satisfaction; moreover any feasible power allocation $x$, (i.e. $x \in \Pi$) is preferred to any  non-feasible power allocation denoted by $0$. All these  considerations justify (\textbf{(A5)}). An example of $V_i(x,t_i)$ is
\begin{eqnarray}
U_i\left(\frac{h_{ii}(1){p}_i^1}{\frac{\mathcal{N}_0}{2}+\sum_{j, j\neq i}h_{ji}(1){p}_j^1},\cdots,\frac{h_{ii}(f){p}_i^f}{\frac{\mathcal{N}_0}{2}+\sum_{j, j\neq i}h_{ji}(f){p}_j^f}\right)-t_i, 
\end{eqnarray}
where $\frac{h_{ii}(k){p}_i^k}{\frac{\mathcal{N}_0}{2}+\sum_{j, j\neq i}h_{ji}(k){p}_j^k}$ is the Signal to Interference Ratio (SIR) in frequency band $k$.
This example illustrates the following: (1) A user's utility function may explicitly depend on the channel gains $h_{ji}, j, i \in \textbf{N};$ (2) User $i, i\in \textbf{N},$ must know $h_{ji}, j \in \textbf{N}$, so that he can be able to evaluate the impact of any feasible power profile $x \in \Pi$ that he proposes on his own utility. Thus, we assume that the channel gains $h_{ji}$, $j\in \textbf{N},$ are known to user $i$, and this is true for every user $i$ (\textbf{(A4)}). These channel gains have to be measured before the communication process starts. In the situation where users are cooperative $h_{ji}$ can easily be determined; user $j$ sends a pilot signal of a fixed power to user $i$, user $i$ measures the received power and determines $h_{ji}$. When users are strategic/selfish, the measurement of $h_{ji}$ can not be achieved according to the process described above, because user $j$ may have an incentive to use a pilot signal other than the one agreed beforehand so that he can obtain an advantage over user $i$. In this situation procedures similar to ones described in \cite{tse} (section V) can be used to measure $h_{ji}$; we present a method, different from those proposed in \cite{tse}, for measuring $h_{ji}(\hat{f}),$ $j,i \in \textbf{N}, \hat{f}\in \textbf{F},$  after we discuss all the assumptions. In (\textbf{(A4)}) we further assume that $h_{ji}$ do not change during the communication process. Such an assumption is reasonable when the mobile users move slowly and  the variation of the channel is considerably slower than the duration of the communication process. Assumption (\textbf{(A8)}) is a behavioral one not a restriction on the model. Since according to (\textbf{(A8)}) users are strategic, each user may not want to reveal his own preference over the set of feasible power allocations, thus assumption (\textbf{(A6)}) is reasonable. It is also reasonable to assume that the function space where each user's utility comes from is the same for all users and common knowledge among all users (\textbf{(A7)}). The fact that a user's utility is his private information along with assumption (\textbf{(A8)}) have an immediate impact on the solution/equilibrium concepts that can be used in the game induced by any mechanism. We will address this issue when we define the objective of our problem. Assumption (\textbf{(A7)}) also ensures that each user uses the same quantization set. Furthermore,  it states that each user knows the power available to every other user. The solution methodology presented in this paper works also for the case where every user knows his total available power, has an upper bound on the power available to all other users, and this upper bound is common knowledge among all users. Assumption (\textbf{(A9)}) is necessary for the game form/mechanism proposed in this paper; it ensures that each user interprets consistently the messages he receives from all other users. \\
In addition to the method described in \cite{tse}, another method for determining the gains $h_{ji}(\hat{f})$, $j,i \in \textbf{N}, \hat{f}\in\textbf{F}$ is the following. We assume that the gain $h_{ij}(\hat{f})$ from the transmitter of pair $i$ to the receiver of pair $j$ is the same as $\bar{h}_{ji}(\hat{f})$, the gain from the receiver of pair $j$ to the transmitter of pair $i$ for all $i, j \in \textbf{N}$ and  $\hat{f}\in \textbf{F}$. Before the power allocation and spectrum sharing process starts, the $(N+1)^{th}$ agent asks transmitter $i$ and receiver $j$ to communicate with one another at frequency $\hat{f}$ by using a fixed power $\bar{p}$, and to report to him their received powers. This communication process takes place as follows: First transmitter $i$ sends a message with power $\bar{p}$ at frequency $\hat{f}$ to receiver $j$; then receiver $j$ sends a message with power $\bar{p}$ at frequency $\hat{f}$ to transmitter $i$; finally transmitter $i$ and receiver $j$ report their received power to the $(N+1)^{th}$ agent. This process is sequentially repeated between transmitter $i$ and receiver $j$ for all frequencies $\hat{f}\in \textbf{F}$. After transmitter $i$ and receiver $j$ complete the above-described communication process, the same process is repeated  sequentially for all transmitter-receiver pairs $(k,l), \ k,l \in \textbf{N}$, at all frequencies $\hat{f}\in\textbf{F}$. The $(N+1)^{th}$ agents collects all the reports generated by the process described above. If the reports of any transmitter $i$ and receiver $j$ ($i \neq j, \ i,j \in \textbf{N}$) differ at any frequency $\hat{f}\in \textbf{F}$, then user $i$ and user $j$ are not allowed to participate in the power allocation and spectrum sharing process. \\
The above-described method for determining $h_{ji}(\hat{f}), \ i,j \in \textbf{N}, \hat{f}\in\textbf{F}$, provides an incentive to user $i, i\in \textbf{N},$ to follow/obey its rules if user $i$ does better by participating in the power allocation and spectrum sharing process than by not participating in it. Consequently, the method proposed for determining $h_{ji}(\hat{f}),\ i,j \in \textbf{N}, \hat{f}\in\textbf{F}$, will work if the game form we propose is individually rational. In this paper we prove that individual rationality is one of the properties of the proposed game form.  
%%%%%%%%%%%%%%%%%%%%%%%%%%%%%%%%%%%%%%%%%%%%%%%%%%%%%%%%%%%%%%%%%%%%%%%%%%%%%%%%%%%%%%%%%%%%5
%%%%%%%%%%%%%%%%%%%%%%%%%%%%%%%%%%%%%%%%%%%%%%%%%%%%%%%%%%%%%%%%%%%%%%%%%%%%%%%%%%%%%%%%%%%%%%%

\end{subsection}
\begin{subsection}{Objective}
\label{000}
The objective is to determine a game form/mechanism that has the following features,
\begin{itemize}
	\item (\textbf{P1}) For any realization $(V_1, V_2,\cdots , V_N,\textbf{Q},\bar{W}) \in \mathcal{V}^N\times \mathcal{Q}\times \mathcal{W}$ all Nash equilibria of the game induced by the game form/mechanism result in  allocations that are weakly Pareto optimal (\cite{tian} pg. 265).
	\item (\textbf{P2}) For every realization $(V_1, V_2,\cdots , V_N,\textbf{Q},\bar{W}) \in \mathcal{V}^N\times \mathcal{Q}\times \mathcal{W}$ the users voluntarily participate in the game induced by the game form/mechanism, i.e, the mechanism is individually rational.  
	\item (\textbf{P3}) For every realization $(V_1, V_2,\cdots , V_N,\textbf{Q},\bar{W}) \in \mathcal{V}^N\times \mathcal{Q}\times \mathcal{W}$, $\sum_{i\in \textbf{N}}t_i(\textbf{m})=0$, where $\textbf{m}$ is  any outcome of the game induced by the game form. That is, for any outcome $\textbf{m}$ of the game we have a balanced budget.
	% ($\textbf{m}$ may or not be a NE), and $\delta\geq 0$ is the total fee (amount of money) that must collectively be paid by the users so that communication  can take place).
\end{itemize}
We follow the philosophy of implementation theory (\cite{tian} Ch. 15) for the specification of our game form. We refer  the reader to  \cite{kakh} for a description of the key ideas of implementation theory and their use in the context of communication networks. In the next section we present a game form/mechanism that achieves the above objectives. However, before we proceed we present a brief clarification on the interpretation of Nash equilibria. Nash equiliria describe strategic behavior in games of complete information. Since in our model the users' utilities are their private information, the resulting game is not one of complete information. We can have a game of complete infromation by increasing the message/strategy space following Maskin's approach \cite{maskin}. However, such an approach would result in an infinite dimensional message/strategy  space for the corresponding game. We don't follow Maskin's approach; instead we adopt the philosophy of \cite{1972,kakh,riter,ledyard2}. Specifically, by quoting \cite{riter}, 
\begin{quote}
\textit{''we interpret our analysis as applying to an unspecified
(message exchange) process in which users grope their way to a
stationary message and in which the Nash property is a necessary
condition for stationarity''}. 
\end{quote}
\end{subsection}

%%%%%%%%%%%%%%%%%%%%%%%%%%%%%%%%%%%%%%%%%%%%%%%%%%%%%%%%%%%%%%%%%%%%%%%%%%%%%%%%%%%%%%%%%%%%%%%%%%%%%%%%%%%%%%%%%%%%%%%%%%%%%%%%%%%%%%%%%%%%%%%%%%%%%%%%%%%%%%%%%%%%%%%%%%%%%%%%%%%%%%%%%%%%%%%%%%%%%%%%%%%%%%%%%%%%%%%%%%%%%%%%%%%%%%%%%%%%%%%%%%%%%%%%%%%%%%%%%%%%%%%%%%%%%%%%%%%%%%%%%%%%%%%%%%%%%%%%%%%%%%%%%%%%%%%%%%%%%%%%%%%%%%%%%%%%%%%%%%%%%%%%%%%%%%%%%%%%%%%%%%%%%%%%%%%%%%%%%%%%%%%%%%%%%%%%%%%%%%%%%%%%%%%%%%%%%%%%%%%%%%%%%%%%%%%%%%%%%%%%%%%%%%%%%%%%%%%%%%%%%%%%%%%%%%%%%%%%%%%%%%%%%%%%%%%%%%%%%%%%%%%%%%%%%%%%%%%%%
\end{section}
%%%%%%%%%%%%%%%%%%%%%%%%%%%%%%%%%%%%%%%%%
%%%%%%%%%%%%%%%%%%%%%%%%%%%%%%%%%%%%%%%%%%%%%%%%%%%%%%%%%%%%%%%%%%%%%%%%%%%%%%%%%%%%%%%%%%%%%%%%%%%%%%%%%%%%%%%%%%%%%%%%%%%%%%%%%%%%%%%%%%%%%%%%%%%%%%%%%%%%%%%%%%%%%%%%5
%%%%%%%%%%%%%%%%%%%%%%%%%%%%%%%%%%%%%%%%%%%%%%%%%%%%%%%%%%%%%%%%%%%%%%%%%%%%%%%%%%%%%%%%%%%%%%%%%%%%%%%%%%%%%%%%%%%%%%%%%%%%%%%%%%%%%%%%%%%%%%%%%%%%%%%%%%%%%%%%%%%%%%%%%

\begin{section}{A Mechanism for Power Allocation and Spectrum Sharing}
\label{11}
For the decentralized problem formulated in section \ref{00} we propose a game form/mechanism the components of which are described as follows.\\
\textbf{Message space} \ $\mathcal{M}:=\mathcal{M}_1 \times \mathcal{M}_2 \times \cdots \mathcal{M}_N$: The message/strategy space for user $i, i=1,2,\cdots,N,$ is given by $\mathcal{M}_i \subseteq \mathbb{Z} \times \mathbb{R}_+$, where $\mathbb{Z}$ and $\mathbb{R}_+$ are the sets of integers and non-negative real numbers, respectively. Specifically, a message of user $i$ is of the form, $\textbf{m}_i=(n_i,\pi_i)$ where $n_i \in \mathbb{Z}$ and $\pi_i\in \mathbb{R}_+$.\\
The meaning of the message space is the following. The component $n_i$ represents the power profile proposed by user $i$; the component $\pi_i$ denotes the price per unit of power user $i$ is willing to pay per unit of the power profile $n_i$. The message $n_i$ belongs to an extended set $\mathbb{Z}$ of power profiles. Every element/integer in $\mathbb{Z}-\Pi$ corresponds to a  power profile that is non-feasible. Working with such an extended  set of power profiles does not alter the solution of the original problem since, as we show in section \ref{33}, all Nash equilibria of the game induced by the proposed mechanism correspond to  feasible power allocations.\\
%\begin{eqnarray}
%\textbf{m}_i=(n_i,\pi_i).
%\end{eqnarray} 
%where $n_i$ denotes an integer number and $\pi_i$ denotes  the price that user $i$ is willing to pay  
\textbf{Outcome function} \ $\hbar$: The outcome function $\hbar$ is given by, $\hbar: \mathcal{M} \rightarrow \mathbb{N}\times \mathbb{R}^N$. 
%\begin{eqnarray}
%\hbar: \mathcal{M} \rightarrow \mathbb{N}\times \mathbb{R}^N
%\end{eqnarray} 
and is defined as follows. For any $\textbf{m}:=(\textbf{m}_1,\textbf{m}_2,\cdots,\textbf{m}_N)\in \mathcal{M}$,
\begin{eqnarray}
\hbar(\textbf{m})&=&\hbar(\textbf{m}_1,\textbf{m}_2,\cdots,\textbf{m}_N)=\left(\left[\texttt{int}\left(\frac{\sum_{i=1}^Nn_i}{N}\right)\right],t_1(\textbf{m}),\cdots,t_N(\textbf{m})\right). \nonumber
\end{eqnarray}
where $\texttt{int}\left(\frac{\sum_{i=1}^Nn_i}{N}\right)$ is the integer number closest to $\left(\frac{\sum_{i=1}^Nn_i}{N}\right)$ and

 %and
\begin{eqnarray}
\label{qi1}
\left[\texttt{int}\left(\frac{\sum_{i=1}^Nn_i}{N}\right)\right] = \left\{ \begin{array}{ll}
        {\texttt{int}\left(\frac{\sum_{i=1}^Nn_i}{N}\right)}, & \mbox{if $\texttt{int}\left(\frac{\sum_{i=1}^Nn_i}{N}\right) \in \{1,2,\cdots,G_N\}$};\\
        0, & \mbox{otherwise}.\end{array} \right.
\end{eqnarray}
%That is, if $\texttt{int}\left(\frac{\sum_{i=1}^Nn_i}{N}\right)\in \{1,2,\cdots,G_N\}$,   then $\texttt{P}_{\texttt{int}\left(\frac{\sum_{i=1}^Nn_i}{N}\right)}$ is a feasible power profile that will be implemented; on the other hand, if  $\texttt{int}\left(\frac{\sum_{i=1}^Nn_i}{N}\right) \notin \{1,2,\cdots,G_N\}$ no communication take place. 
  The component  $t_i, i=1,2,\cdots,N,$ describes the tax (subsidy) that user $i$  pays (receives). The tax(subsidy) for every user is defined as follows,  
\begin{eqnarray}
\label{tax}
t_i(\textbf{m})&=&\left\{\texttt{int}\left(\frac{\sum_{i=1}^Nn_i}{N}\right)\left[\frac{\pi_{i+1}-\pi_{i+2}}{N}\right]+(n_i-n_{i+1})^2\pi_i-(n_{i+1}-n_{i+2})^2\pi_{i+1}\right\}\nonumber \\
&& \times 1\left\{\texttt{int}\left(\frac{\sum_{i=1}^Nn_i}{N}\right)\in \{1,2,\cdots,G_N\} \right\}
\end{eqnarray}
%where
%\begin{eqnarray}
%\Xi_1&:=&\texttt{int}\left(\frac{\sum_{i=1}^Nn_i}{N}\right)\left[\frac{\pi_{i+1}-\pi_{i+2}}{N}\right] \\
%\Xi_2&:=&(n_i-n_{i+1})^2\pi_i\\
%\Xi_3&:=&-(n_{i+1}-n_{i+2})^2\pi_{i+1}
%\end{eqnarray}
where $1\{A\}$ denotes the indicator function of event $A$, that is, $1\{A\}=1$ if $A$ is true and $1\{A\}=0$ otherwise, and $N+1$ and $N+2$ are to be interpreted as $1$ and $2$, respectively.
\end{section} 
%%%%%%%%%%%%%%%%%%%%%%%%%%%%%%%%%%%%
\begin{section}{Interpretation of the Mechanism}
\label{2222}
As pointed out in section \ref{00}, the design of an efficient resource allocation mechanism has to achieve the following goals. (i) It must induce strategic users to voluntarily participate in the allocation process. (ii) It must induce strategic users to follow its operational rules. (iii) It must result in weakly Pareto optimal allocations at all equilibria of the induced game. (iv) It must result in a balanced budget at all equilibria and off equilibrium. \\
To achieve these goals we propose the tax incentive function described by \eqref{tax}. This function consists of three components, $\Xi_1, \Xi_2$  and $\Xi_3$, that is, 
\begin{eqnarray}
t_i(\textbf{m})&=&\underbrace{\texttt{int}\left(\frac{\sum_{i=1}^Nn_i}{N}\right)\left[\frac{\pi_{i+1}-\pi_{i+2}}{N}\right]}_{\Xi_1}
+\underbrace{(n_i-n_{i+1})^2\pi_i}_{\Xi_2}
\underbrace{-(n_{i+1}-n_{i+2})^2\pi_{i+1}}_{\Xi_3}
\end{eqnarray}  
%where 
%\begin{eqnarray}
%&&\Xi_1:=\texttt{int}\left(\frac{\sum_{i=1}^Nn_i}{N}\right)\left[\frac{\pi_{i+1}-\pi_{i+2}}{N}\right]\\
%&&\Xi_2:=(n_i-n_{i+1})^2\pi_i\\
%&&\Xi_3:=-(n_{i+1}-n_{i+2})^2\pi_{i+1}
%\end{eqnarray}
The term $\Xi_1$ specifies the amount that each user must pay for the power profile which is determined by the mechanism. The price per unit of power, $\frac{\pi_{i+1}-\pi_{i+2}}{N}$, paid by user $i, i=1,2,\cdots, N,$ is not controlled by that user.   The terms $\Xi_2$ considered collectively provide an incentive to all users to propose the same power profile. The term $\Xi_3$ is not controlled by user $i$, its goal is to lead to a balanced budget.
\end{section}
%%%%%%%%%%%%%%%%%%%%%%%%%%%%%%%%%%%%%%%
%%%%%%%%%%%%%%%%%%%%%%%%%%%%%%%%%%%%%%%
\begin{section}{Properties of The Mechanism}
\label{33}
We prove the mechanism proposed in section \ref{11} has the  properties (\textbf{P1}), (\textbf{P2}) and (\textbf{P3}) stated in subsection \ref{000} by proceeding as follows. First, we derive a property of every NE of the game induced by the mechanism proposed in section \ref{00}, (Lemma \ref{lem0}); based on this result we determine the form of the tax(subsidy) at all  Nash equilibria. Then, we show that every NE of the game induced by the mechanism proposed results in a feasible allocation, (Lemma \ref{lem010}). Afterward, we prove that the  proposed mechanism is always budget balanced, (Lemma \ref{lem1}). Subsequently  we show that users voluntarily participate\footnote{In proving voluntary participation, we follow the philosophy presented by Fudenberg and Tirole (\cite{fudenberg} p. 244-245), namely that mechanism design is a three step game. In the first step the designer designs a mechanism, in the second step agents simultaneously accept or reject the mechanism and in the third step agents who accept the mechanism play the game specified by the mechanism. As a consequence of this philosophy non-participation is not a NE.} in the game, by proving that the utility they receive at all NE is greater than or equal to zero, which is the utility they receive by not participating in the power allocation and spectrum sharing process, (Lemma \ref{lem2}). Finally, we show that every NE of the game induced by the  mechanism proposed in section \ref{11} results in a Lindahl equilibrium (\cite{tian} Ch. 12.4.2); that is, every NE results in a weakly Pareto optimal allocation (Theorem \ref{th1}). Furthermore, we prove that every Lindahl equilibrium can be associated with  a NE of the game induced by the proposed mechanism  in section \ref{11}, (Theorem \ref{th2}). \\
We now proceed to prove the above-stated properties.       
\begin{lemma}
\label{lem0}
Let $\textbf{m}^*$ be a NE of the game induced by the proposed mechanism. Then for every $i, i=1,2,\cdots,N$, we have
%For every $NE$  $\textbf{m}^*$ of the  mechanism proposed in section \ref{11}, we have
\begin{eqnarray}
\label{23}
(n_i^*-n_{i+1}^*)^2\pi_{i}^*=0.
\end{eqnarray} 
\end{lemma}
\begin{proof}
Since $\textbf{m}^*=((n_1^*,\pi_1^*),(n_2^*,\pi_2^*),\cdots,(n_N^*,\pi_N^*))$ is a $NE$,  the following holds for every $i, i=1,2,\cdots,N$,
\begin{eqnarray}
\label{nash}
V_i\left(\left[\texttt{int}\left(\frac{\sum_{k=1}^N n_k^*}{N}\right)\right],t_i(\textbf{m}^*)\right) \geq  V_i\left(\left[\texttt{int}\left(\frac{\sum_{\substack{k=1 \\ k\neq i}}^N n_k^*+n_i}{N}\right)\right],t_i(\textbf{m}_{i},\textbf{m}_{-i}^*)\right) \nonumber \\
\forall \ \textbf{m}_i \in \mathcal{M}_i, 
\end{eqnarray}
where $\textbf{m}_{-i}:=(\textbf{m}_1,\textbf{m}_2,\cdots,\textbf{m}_{i-1},\textbf{m}_{i+1},\cdots,\textbf{m}_{N})$. \\
Set $n_i$ equal to $n_i^*$; then for every $\pi_i \geq 0$ Eq.  \eqref{nash} along with \eqref{tax} imply
\begin{eqnarray}
\label{13}
V_i\left(\left[\texttt{int}\left(\frac{\sum_{k=1}^N n_k^*}{N}\right)\right],t_i(\textbf{m}^*)\right)\geq 
  V_i\left(\left[\texttt{int}\left(\frac{\sum_{k=1}^N n_k^*}{N}\right)\right],t_i(\textbf{m}_{i},\textbf{m}_{-i}^*)\right)
\end{eqnarray}
where
\begin{eqnarray} 
\label{ee8}
%&&x^*:=\left[\texttt{int}\left(\frac{\sum_{k=1}^N n_k^*}{N}\right)\right]\\
t_i(\textbf{m}^*)&=& \left\{\texttt{int}\left(\frac{\sum_{k=1}^N n_k^*}{N}\right)\left[\frac{\pi_{i+1}^*-\pi_{i+2}^*}{N}\right]+(n_i^*-n_{i+1}^*)^2\pi_i^*-(n_{i+1}^*-n_{i+2}^*)^2\pi_{i+1}^*\right\} \nonumber \\
&&\times 1\left\{\texttt{int}\left(\frac{\sum_{i=1}^Nn_i^*}{N}\right)\in \{1,2,\cdots,G_N\} \right\}   
\end{eqnarray}
\begin{eqnarray}
\label{ee9}
t_i((n_i^*,\pi_i),\textbf{m}_{-i}^*)&=& \left\{\texttt{int}\left(\frac{\sum_{k=1}^N n_k^*}{N}\right)\left[\frac{\pi_{i+1}^*-\pi_{i+2}^*}{N}\right]+(n_i^*-n_{i+1}^*)^2\pi_i-(n_{i+1}^*-n_{i+2}^*)^2\pi_{i+1}^*\right\} \nonumber \\
&&\times 1\left\{\texttt{int}\left(\frac{\sum_{i=1}^Nn_i^*}{N}\right)\in \{1,2,\cdots,G_N\} \right\}   
\end{eqnarray}
%and
%\begin{eqnarray}
%\Xi_1^*&:=&\texttt{int}\left(\frac{\sum_{k=1}^N n_k^*}{N}\right)\left[\frac{\pi_{i+1}^*-\pi_{i+2}^*}{N}\right]\\
%\Xi_2^*&:=&(n_i^*-n_{i+1}^*)^2\pi_i^*\\
%\Xi_3^*&:=&-(n_{i+1}^*-n_{i+2}^*)^2\pi_{i+1}^*\\
%\Upsilon_2&:=&(n_i^*-n_{i+1}^*)^2\pi_i
%\end{eqnarray} 
Since $V_i$ is decreasing in $t_i$  Eq. \eqref{13} along with \eqref{ee8} and \eqref{ee9} yield 
\begin{eqnarray}
\pi_{i}^*(n_i^*-n_{i+1}^*)^2 \leq \pi_{i}(n_i^*-n_{i+1}^*)^2 \quad \forall \ \pi_i \geq 0. 
\end{eqnarray}
Therefore, $\pi_{i}^*(n_i^*-n_{i+1}^*)^2=0$ for every $i, i=1,2,\cdots,N$, at every NE $\textbf{m}^*$. 
\end{proof}
An immediate consequence of Lemma \ref{lem0} is the following.
%\label{cor0}
At every NE $\textbf{m}^*$ of the mechanism the tax function $t(\textbf{m}^*)$ has the form
\begin{eqnarray}
\label{18}
t_i(\textbf{m}^*)=1\left\{\texttt{int}\left(\frac{\sum_{i=1}^Nn_i^*}{N}\right)\in \{1,2,\cdots,G_N\} \right\}
\times \texttt{int}\left(\frac{\sum_{k=1}^N n_k^*}{N}\right)\left[\frac{\pi_{i+1}^*-\pi_{i+2}^*}{N}\right]. 
\end{eqnarray}
In the following lemma, we show that every NE of the game induced by the proposed mechanism is feasible.
\begin{lemma}
\label{lem010}
Every NE of the game induced by the proposed mechanism results in a feasible allocation.
\end{lemma}
\begin{proof}
We prove the assertion of the lemma by contradiction. Let $\textbf{m}^*$ be a NE for the game induced by the mechanism.  Suppose $\textbf{m}^*$ does not result in a feasible allocation, i.e., $\texttt{int}\left(\frac{\sum_{i=1}^Nn_i^*}{N}\right) \notin \{1,2,3,\cdots,G_N\}$. Then $\left[\texttt{int}\left(\frac{\sum_{j=1}^Nn_j^*}{N}\right)\right]=0$. Since $\sum_{j=1}^N(\pi_{j+1}^*-\pi_{j+2}^*)=0$,  there exists $i, i\in \{1,2,\cdots,N\}$, such that 
\begin{eqnarray}
\label{432}
\pi_{i+1}^*-\pi_{i+2}^* \leq 0.
\end{eqnarray}
Keep $\textbf{m}_{-i}^*$ fixed and define $\textbf{m}_i=(n_i,\pi_i)$ as follows; set $\pi_i=0$, and choose $n_i$  such that $\texttt{int}\left(\frac{\sum_{\substack{j=1, j\neq i}}^Nn_j^*+n_i}{N}\right)\in \{1,2,\cdots,G_N\}.$
\begin{comment}
\begin{eqnarray}
\label{431}
\pi_i=0,
\end{eqnarray}
and choose $n_i$  such that  
\begin{eqnarray}
\label{khoda}
\texttt{int}\left(\frac{\sum_{\substack{j=1, j\neq i}}^Nn_j^*+n_i}{N}\right)\in \{1,2,\cdots,G_N\}.
\end{eqnarray} 
\end{comment}
Now,  \eqref{tax}  yield that 
\begin{eqnarray}
\label{430}
t_i(\textbf{m}_i,\textbf{m}_{-i}^*)\leq 0.
\end{eqnarray}  
Equations  (12) and \eqref{430} along with Lemma \ref{lem0} and  assumption (A5)  result in 
\begin{eqnarray}
\label{428}
V_i(0,0)=V_i\left(\left[\texttt{int}\left(\frac{\sum_{j=1}^Nn_j^*}{N}\right)\right],t_i(\textbf{m}^*)\right)<  V_i\left(\left[\texttt{int}\left(\frac{\sum_{\substack{j=1, j\neq i}}^Nn_j^*+n_i}{N}\right)\right],t_i(\textbf{m}_i,\textbf{m}_{-i}^*)\right).
\end{eqnarray}
But \eqref{428} is in contradiction with the fact that $\textbf{m}^*$ is a NE. Therefore, every NE of the game induced by the proposed mechanism results in a feasible allocation.
\end{proof}
In the following lemma, we show that the proposed mechanism is always budget balanced.
\begin{lemma}
\label{lem1}
The proposed mechanism is always budget balanced.
\end{lemma} 
\begin{proof}
%Since the collective cost incurred by the users for communication with one another is equal to $\delta$, $\delta\geq 0$ and constant,
To have a balanced budget it is necessary and sufficient to satisfy $\sum_{i=1}^N t_i(\textbf{m}_i)=0$. It is easy to see that it always holds since from \eqref{tax} we have 
\begin{eqnarray}
\label{balanced tax}
\sum_{i=1}^N t_i(\textbf{m})&=& \sum_{i=1}^N\left\{\texttt{int}\left(\frac{\sum_{i=1}^Nn_i}{N}\right)\left[\frac{\pi_{i+1}-\pi_{i+2}}{N}\right]+(n_i-n_{i+1})^2\pi_i-(n_{i+1}-n_{i+2})^2\pi_{i+1}\right\}
\nonumber \\
&=&0.
\end{eqnarray} 
Equality \eqref{balanced tax} holds, since $\sum_{i=1}^N(\pi_{i+1}-\pi_{i+2})=\sum_{i=1}^N\left((n_i-n_{i+1})^2\pi_i-(n_{i+1}-n_{i+2})^2\pi_{i+1}\right)=0.$
\end{proof}
The next result asserts that the mechanism/game form proposed in section \ref{11} is individually rational.
\begin{lemma}
\label{lem2}
The game form specified in section \ref{11} is individually rational, i.e.,  at every NE $\textbf{m}^*$ the corresponding allocation $\left(\texttt{int}\left(\frac{\sum_{i=1}^Nn_i^*}{N}\right),t_1(\textbf{m}^*),t_2(\textbf{m}^*),\cdots,t_N(\textbf{m}^*)\right)$  is weakly preferred by all users to their initial endowment $(0,0)$.
\end{lemma}

\begin{proof}
We need to show that $V_i\left(\texttt{int}\left(\frac{\sum_{i=1}^Nn_i^*}{N}\right),t_i(\textbf{m}^*)\right) \geq V_i(0,0)=0$ for every $i,i=1,2,\cdots,N.$ By the property of every NE  it follows that for every $i \in \textbf{N}$ and $(n_i,\pi_i)\in \mathcal{M}_i$
\begin{eqnarray}
\label{nash1}
V_i\left(\texttt{int}\left(\frac{\sum_{k=1}^Nn_k^*}{N}\right),t_i(\textbf{m}^*)\right) \geq 
 V_i\left(\left[\texttt{int}\left(\frac{\sum_{\substack{k=1 \\ k\neq i}}^N n_k^*+n_i}{N}\right)\right],t_i((n_i,\pi_i),\textbf{m}_{-i}^*)\right).
\end{eqnarray}
Choosing $n_i$ sufficiently large so that $\texttt{int}\left(\frac{\sum_{\substack{k=1 \\ k\neq i}}^N n_k^*+{n}_i}{N}\right) \notin \{1,2,\cdots,G_N\},$ gives
\begin{eqnarray}
\label{466}
\left[\texttt{int}\left(\frac{\sum_{\substack{k=1 \\ k\neq i}}^N n_k^*+\hat{n}_i}{N}\right)\right]=0,
\end{eqnarray}
because of  \eqref{qi1}, and 
\begin{eqnarray}
\label{467}
t_i((n_i,\pi_i),\textbf{m}_{-i}^*)=0.
\end{eqnarray}
because of  \eqref{tax}.  Consequently, \eqref{466} and \eqref{467} establish that 

%So, it is enough to find $(\bar{n}_i,\bar{\pi}_i)\in \mathcal{M}_i$ so that 
%\begin{eqnarray}
%V_i\left(\left[\texttt{int}\left(\frac{\sum_{\substack{k=1 \\ k\neq i}}^N n_k^*+n_i}{N}\right)\right],t_i((n_i,\pi_i),\textbf{m}_{-i}^*)\right) \geq 0..
%\end{eqnarray} 
%We set
%\begin{eqnarray}
%\label{45} 
%\hat{\pi}_i=0,
%\end{eqnarray} 
%and choose $\hat{n}_i$  sufficiently large so that 
%\begin{eqnarray}
%\label{466}
%\texttt{int}\left(\frac{\sum_{\substack{k=1 \\ k\neq i}}^N n_k^*+\hat{n}_i}{N}\right) \notin \{1,2,\cdots,G_N\}.
%\end{eqnarray}
%Then,
%\begin{eqnarray}
%\label{46}
% \left[\texttt{int}\left(\frac{\sum_{\substack{k=1 \\ k\neq i}}^N n_k^*+\hat{n}_i}{N}\right)\right]=0;
%\end{eqnarray}
%furthermore, from Lemma \ref{lem0} we have
%\begin{eqnarray}
%\label{47}
%(n_{i+1}^*-n_{i+2}^*)^2\pi_{i+1}^*=0.
%\end{eqnarray}
 
%Therefore, \eqref{tax} along with \eqref{45}, \eqref{46} and \eqref{47} imply 
%\begin{eqnarray}
%\label{48}
%t_i((\hat{n}_i,\hat{\pi}_i),\textbf{m}_{-i}^*)=0,
%\end{eqnarray}
%and consequently, \eqref{nash1} together with \eqref{46} and \eqref{48} result in 
\begin{eqnarray}
V_i\left(\texttt{int}\left(\frac{\sum_{k=1}^Nn_k^*}{N}\right),t_i(\textbf{m}^*)\right) \geq V_i(0,0)=0.
\end{eqnarray}
\end{proof}

In the following theorem, we show that every NE of the game induced by the mechanism proposed in section \ref{11} results in a Lindahl equilibrium.

\begin{theorem}
\label{th1}
Suppose that an allocation $\Psi_{\textbf{m}}$, for any $\textbf{m}\in \mathcal{M}$, is determined as follows $$\Psi_{\textbf{m}}:=\left(\Lambda(\textbf{m}),t_1(\textbf{m}),\cdots,t_N(\textbf{m}), L_1,\cdots,L_N\right)$$ where  $\Lambda(\textbf{m}):=\left[\texttt{int}\left(\frac{\sum_{k=1}^Nn_k}{N}\right)\right]$, for each $i, i=1,2,\cdots,N,$ $t_i(\textbf{m})$ is defined by \eqref{tax}, and 
\begin{eqnarray}
\label{50}
L_i:=\frac{\pi_{i+1}-\pi_{i+2}}{N}.
\end{eqnarray}
Then $\Psi_{\textbf{m}^*}$ is a Lindahl equilibrium corresponding to the NE $$\textbf{m}^*=\left((n_1^*,\pi_1^*),(n_2^*,\pi_2^*),\cdots,(n_N^*,\pi_N^*)\right)$$ of the game induced by the proposed mechanism. 
\end{theorem}
\begin{proof}
$\Psi_{\textbf{m}^*}$ defines a Lindahl equilibrium if it satisfies the following three conditions (\cite{tian} Ch. 12.4.2)
\begin{enumerate}

	\item (\textit{C1}): $\sum_{i=1}^NL_i^*=0.$
	\item (\textit{C2}): $\sum_{i=1}^Nt_i(\textbf{m}^*)=0$.
	\item (\textit{C3}): For all $i,i=1,2,\cdots,N,$ $\left(\texttt{int}\left(\frac{\sum_{k=1}^Nn_k^*}{N}\right),t_i(\textbf{m}^*)\right)$ is a solution of the following optimization problem:
	\begin{eqnarray}
	\label{1000}
	&\max_{x,t_i} \quad &V_i(x,t_i) \nonumber \\
	&\mbox{subject} \ \mbox{to} \quad &x \ L_i^*=t_i \nonumber \\
	& \quad  \quad \quad &x \in \Pi, \ t_i \geq 0.
	\end{eqnarray} 
\end{enumerate}
By  simple algebra we can show that conditions 1 and 2 are satisfied. We need to prove that condition 3 is also satisfied. We do this by contradiction. Suppose $\left(\texttt{int}\left(\frac{\sum_{k=1}^Nn_k^*}{N}\right),t_i(\textbf{m}^*)\right)$ is not a solution of the optimization problem defined by \eqref{1000} for all $i$. Then, for some  user $i, i\in \{1,2,\cdots,N\}$, there is a power profile $\zeta \in \Pi$ and $\zeta \neq \texttt{int}\left(\frac{\sum_{k=1}^Nn_k^*}{N}\right)$ such that 
\begin{eqnarray}
\label{70}
V_i\left(\texttt{int}\left(\frac{\sum_{k=1}^Nn_k^*}{N}\right),\left[\texttt{int}\left(\frac{\sum_{k=1}^Nn_k^*}{N}\right)\right] \ L_i^*\right)
<V_i(\zeta,\zeta \ L_i^*).
\end{eqnarray}  
Now choose $\bar{\pi}_i=0$ and $\bar{n}_i=\texttt{int}\left(N \zeta-\sum_{\substack{j=1 \\ j\neq i}}^N n_j^*\right)$. Using Eq. \eqref{tax} and \eqref{23} together with the fact that $\bar{\pi}_i=0$ we obtain
\begin{eqnarray}
\label{71}
t_i((\bar{n}_i,\bar{\pi}_i),\textbf{m}_{-i}^*)=\zeta \left[\frac{\pi_{i+1}^*-\pi_{i+2}^*}{N}\right]=\zeta \ L_i^*.
\end{eqnarray}
Then, because of \eqref{70} and \eqref{71} we get
\begin{eqnarray}
\label{72}
&V_i(\zeta,\zeta \ L_i^*)=V_i\left(\left[\texttt{int}\left(\frac{\sum_{\substack{j=1 \\ j \neq i}}^Nn_j^*+\bar{n}_i}{N}\right)\right],t_i((\bar{n}_i,\bar{\pi}_i),\textbf{m}_{-i}^*)\right) 
 \geq V_i\left(\left[\texttt{int}\left(\frac{\sum_{\substack{j=1 \\ j \neq i}}^Nn_j^*+{n}_i^*}{N}\right)\right],t_i(\textbf{m}^*)\right) \nonumber 
\end{eqnarray}
which is a contradiction, because $$\textbf{m}^*=\left((n_1^*,\pi_1^*),(n_2^*,\pi_2^*),\cdots,(n_N^*,\pi_N^*)\right)$$ is a NE of the game induced by the proposed game form. Consequently, $\left(\texttt{int}\left(\frac{\sum_{k=1}^Nn_k^*}{N}\right),t_i(\textbf{m}^*)\right)$ is a solution of the optimization problem defined by \eqref{1000} for all $i$. Since $\Psi_{\textbf{m}^*}$ satisfies \textit{(C1)-(C3)} it defines  a Lindahl equilibrium. The allocation $\left\{\texttt{int}\left(\frac{\sum_{k=1}^Nn_k^*}{N}\right),t_1(\textbf{m}^*),t_2(\textbf{m}^*),\cdots,t_N(\textbf{m}^*)\right\}$  is also weakly Pareto optimal (\cite{tian} Theorem (12.4.1)).
\end{proof}

Finally,  we establish that any Lindahl equilibrium can be associated with a NE of the game induced by the proposed mechanism.
\begin{theorem}
\label{th2}
Let 
$$
\Psi=\left(\Lambda^{\ell},t_1^{\ell},t_2^{\ell},\cdots,t_N^{\ell}, L_1^{\ell},L_2^{\ell},\cdots,L_N^{\ell}\right)
$$
 be a Lindahl equilibrium. Then, there does exist a NE $\textbf{m}^*$ of the game induced by the proposed mechanism so  that
 \begin{eqnarray}
 \hbar(\textbf{m}^*)=\left(\Lambda^{\ell},t_1^{\ell},t_2^{\ell},\cdots,t_N^{\ell}\right)
 \end{eqnarray}
where for every $i,i=1,2,\cdots,N$, $t_i(\textbf{m}^*)=\Lambda^{\ell} \  L_i^{\ell}.$
%\begin{eqnarray}
%t_i(\textbf{m}^*)=\Lambda^{\ell} \  L_i^{\ell}.
%\end{eqnarray}
\end{theorem}
%%%%%%%%%%%%%%%%%%%%%%%%%%%%%%%%%%%%%%%%%%%%%%
%%%%%%%%%%%%%%%%%%%%%%%%%%%%%%%%%%%%%%%%%%%%
\begin{proof}
Consider the message profile ${\textbf{m}^*}$ such that for every $i, i=1,2,\cdots, N$, $\textbf{m}_i^*=({n}_i^*,{{\pi}}_i^*)$ and,  $\forall i, i=1,2,\cdots,N$, $n_i^*=(\Lambda^{\ell})$ and $\pi_i^*$'s are the solution of the following system of equations, 
\begin{eqnarray}
\label{sys}
L_{1}^{\ell}=\frac{\pi^*_{2}-\pi^*_{3}}{N}, \  L_{2}^{\ell}=\frac{\pi^*_{3}-\pi^*_{4}}{N}, \  \cdots, \ L_{N}^{\ell}=\frac{\pi^*_{1}-\pi^*_{2}}{N}.
\end{eqnarray}
Choosing $\pi_1^*$ sufficiently large guarantees that the following is a feasible solution to \eqref{sys}, i.e., $\pi_i \geq 0, \forall \ i$, $\pi_1^*=\mbox{sufficiently} \ \mbox{large}$, $\pi_2^*=\pi_1^*-L_{N}^{\ell}$ and
\begin{eqnarray}
\label{21}
\pi_i^*=(i-1)\pi_1^*-\left(L_{N}^{\ell}+\sum_{j=1}^{i-2}L_{j}^{\ell}\right) \quad   i, \ 3 \leq i \leq N. \nonumber
\end{eqnarray}
Furthermore, 
\begin{eqnarray}
\label{22}
\Lambda^{\ell}=\left[\texttt{int}\left(\frac{\sum_{k=1}^Nn_k^*}{N}\right)\right]. 
\end{eqnarray}
To complete the proof, we need to prove that ${\textbf{m}}^*$ is a NE of the game induced by the mechanism. For that matter, it is enough to show that, for every $i, i=1,2,\cdots,N,$
\begin{eqnarray}
\label{bbb}
V_i\left(\left[\texttt{int}\left(\frac{\sum_{k=1}^N n_k^*}{N}\right)\right],t_i({\textbf{m}}^*)\right) \geq  V_i\left(\left[\texttt{int}\left(\frac{\sum_{\substack{k=1 \\ k\neq i}}^N {n}_k^*+n_i}{N}\right)\right],t_i({\textbf{m}}_{-i}^*,\textbf{m}_i)\right) \nonumber \\
 \forall \ \textbf{m}_i\in\mathcal{M}.
\end{eqnarray}
Equation \eqref{18} along with Eqs. \eqref{sys} and  \eqref{22} imply $t_i({\textbf{m}^*})=L_{i}^{\ell}\left[\texttt{int}\left(\frac{\sum_{k=1}^N {n}_k^*}{N}\right)\right].$
%\begin{eqnarray}
%\label{b}
%t_i({\textbf{m}^*})=L_{i}^{\ell}\left[\texttt{int}\left(\frac{\sum_{k=1}^N {n}_k^*}{N}\right)\right].
%\end{eqnarray}
Furthermore, positivity of $({n}_{i+1}^*-n_i)^2\pi_i$ together with fact that $V_i$ is decreasing in $t_i$ give that 
\begin{eqnarray}
\label{c}
V_i\left(\xi,L_{i}^{\ell}\xi\right) \geq V_i\left(\xi,L_{i}^{\ell}\xi+({n}_{i+1}^*-n_i)^2\pi_i\right) \quad \forall \ \xi, \ \xi \in \Pi. \nonumber 
\end{eqnarray}
Moreover, since $\Psi$ is a Lindahl equilibrium,  \textit{(C3)} implies that the following  holds for every $i, i=1,2,\cdots,N,$ 
\begin{eqnarray}
\label{16}
V_i\left(\Lambda^{\ell},L_{i}^{\ell}\Lambda^{\ell}\right) \geq V_i\left(\xi,L_{i}^{\ell} \ \xi \right) \quad \forall \xi \in \Pi, 
\end{eqnarray}  
Consequently, the fact that $t_i({\textbf{m}^*})=L_{i}^{\ell}\left[\texttt{int}\left(\frac{\sum_{k=1}^N {n}_k^*}{N}\right)\right]$  along  with  \eqref{bbb} and \eqref{c} result in 
\begin{eqnarray}
V_i\left(\left[\texttt{int}\left(\frac{\sum_{k=1}^N{n}_k^*}{N}\right)\right],t_i({\textbf{{m}}}^*)\right)&=&V_i\left(\left[\texttt{int}\left(\frac{\sum_{k=1}^Nn_k^*}{N}\right)\right],L_{i}^{\ell}\left[\texttt{int}\left(\frac{\sum_{k=1}^Nn_k^*}{N}\right)\right]\right)\nonumber \\
&\geq & V_i\left(\xi,L_{i}^{\ell}\xi\right) \nonumber \\
&\geq & V_i\left(\xi,L_{i}^{\ell}\xi+({n}_{i+1}^*-n_i)^2\pi_i\right) \quad \forall \ \xi, \xi\in \Pi \nonumber \\
&=&V_i\left(\left[\texttt{int}\left(\frac{\sum_{\substack{j=1 \\ j\neq i}}^N{n}_j^*+n_i}{N}\right)\right],t_i(\textbf{m}_i,{\textbf{{m}}}_{-i}^{*})\right) \quad \forall \ \textbf{m}_i\in\mathcal{M}. \nonumber
\end{eqnarray}
Therefore ${\textbf{m}}^*$ is a NE of the game induced by the proposed mechanism.
%Eq. \eqref{21} along with Eq. \eqref{22} imply that $\bar{\textbf{m}}$ is  
\end{proof}
\end{section}
%%%%%%%%%%%%%%%%%%%%%%%%%%%%%%%%%%
\begin{section}{Related Work}
\label{dis}
The game form/mechanism we proposed and  analyzed in this paper has, for any realization $(V_1, V_2,\cdots , V_N,\textbf{Q},\bar{W}) \in \mathcal{V}^N\times \mathcal{Q}\times \mathcal{W}$,  the following properties:
\begin{itemize}
 \item (\textbf{P1}) It is budget-balanced at every NE of the induced game as well as off equilibrium.
 \item (\textbf{P2}) It is individually rational, i.e.,  every user voluntarily participates  in the induced game.
 \item (\textbf{P3}) Every NE of the induced game results in a Lindahl equilibrium; thus, the allocations corresponding to all NE are weakly Pareto optimal. Conversely,  every Lindahl equilibrium results in a NE of the  game induced by the mechanism.
	\end{itemize}
Our proposed  game form/mechanism  achieves all the above desirable properties without any assumption about,  concavity, differentiability, monotonicity or quasi-linearity of the users' utility functions (the only assumption made on the users' utility functions is ((A5))).\\
The results presented in this paper are distinctly different from those currently existing in the literature for the reasons we explain below.\\
Most of previous work within the context of competitive power allocation games has investigated Gaussian interference games (\cite{cioffi, tse}), that is, situations where the users operate in a Gaussian noise environment. In a Gaussian interference game, every user can spread a fixed amount of power arbitrarily across a continuous bandwidth, and attempts to maximize its total rate over all possible power allocation strategies. In \cite{cioffi}, the authors proved the 
existence and uniqueness of a NE for a two-player version of the game, and provided an iterative water-filling
algorithm  to obtain the NE. This work was extended in \cite{tse}, where it  was shown that the aforementioned pure NE can be
quite inefficient, but by playing an infinitely  repeated game system performance can be improved. Our results are different from those in \cite{cioffi, tse} because : (i) The users are allowed to transmit at a discrete set of frequencies, and the power allocated at each frequency most be chosen from a discrete set. (ii)  The unique pure NE of the one-stage game in \cite{cioffi, tse} does not necessarily  result in a weakly Pareto optimal allocations. (iii) Most of the NE of the repeated game in \cite{tse} result in allocations that are not weakly Pareto optimal.

In \cite{berry}, the authors presented a market-based model for situations where  every user can only use one or more than one frequency bands, and the game induced by their proposed game form is super modular. They developed/presented a distributed best response algorithm that converges to a NE. However, in general the Nash equilibria of the game induced by their mechanism  are not efficient, that is, they do not always result in optimal centralized power allocations, or weakly Pareto optimal allocations. 

In \cite{liu} the authors investigated the case where all users have the same utility function and each user can only use one frequency band. They proved the existence of a NE in the game resulting from the above assumptions. The NE is, in general, not efficient. The results in \cite{liu} critically depend on the fact that the users' utilities are identical and monotonic; these constraints are not present in our model.

The game form/mechanism we have proposed/analyzed in this paper is in the category of the mechanisms that economists created for  public good problems \cite{walker,hurwize,ledyard}, but it is distinctly different form all of them because the allocation spaces $\textbf{F}$ and $\textbf{Q}$ in our formulation are discrete. To the best of our knowledge, the game form we presented in this paper is the first mechanism for power allocation and spectrum sharing in multi-user, multi-channel systems with strategic users  that  achieves all three desirable properties (\textbf{P1})-(\textbf{P3}). Furthermore, we do not impose any 
assumption about, concavity, differentiability, monotonicity or quasi-linearity of the users' utility
functions.
\end{section}
\begin{section}{Conclusion}
\label{con}
We have discovered a game form for power allocation and spectrum sharing in multi-user, multi-channel systems with strategic users that possesses several desirable properties. We have performed an equilibrium analysis of the game form. Currently we do not have an algorithm (tatonnement process) for the computation of the Nash equilibria of the mechanism. The  discovery of such an algorithm is an important open problem. 
In this paper we have investigated a static allocation problem with strategic users. The discovery of mechanisms for the dynamic analogue of the problem presented in this paper is an important open problem.  
\end{section}
\subsection*{Acknowledgments}
This research was supported in part by NSF Grant \texttt{CCR-0325571} and NASA Grant \texttt{NNX09AE91G}. The authors gratefully acknowledge stimulating discussions with A. Anastasopoulos, M. Liu and R. Ghaemi.

%%%%%%%%%%%
\end{document}